\documentclass[reqno]{amsart}
\usepackage{amsmath}
\usepackage{amsthm}
\usepackage{amsfonts}
\usepackage{amssymb}
\usepackage{graphicx}
\usepackage{mathrsfs}
\usepackage{cite}
\usepackage{textcomp}

\title
[Bounds for lattice sums]
{Multiplicative Diophantine approximation and bounds for lattice sums }
\author[M.M. SKRIGANOV]{M.M. SKRIGANOV}

\address{St. Petersburg Department of the Steklov Mathematical Institute 
of the Russian Academy of Sciences, 
27, Fontanka, St.Petersburg, 191023, Russia}

\email{maksim88138813@mail.ru}

\keywords{Diophantine approximation, Geometry of numbers, Lattice sums}

\subjclass[2010]{11J13, 11J87, 11H46, 42B05, 05A18.}

\numberwithin{equation}{section}

\newtheorem{theorem}{Theorem}[section]
\newtheorem{lemma}{Lemma}[section]
\newtheorem{proposition}{Proposition}[section]
\newtheorem{corollary}{Corollary}[section]

\theoremstyle{remark}
\newtheorem{remark}{Remark}[section]
\theoremstyle{remark}
\newtheorem{definition}{Definition}[section]

\def\bfa{\mathbf{a}}
\def\bfb{\mathbf{b}}

\def\bfm{\mathbf{m}}

\def\bfq{\mathbf{q}}

\def\bfu{\mathbf{u}}
\def\bfv{\mathbf{v}}
\def\bfw{\mathbf{w}}
\def\bfx{\mathbf{x}}
\def\bfy{\mathbf{y}}
\def\bfz{\mathbf{z}}

\def\Qq{\mathbb{Q}}
\def\Rr{\mathbb{R}}

\def\Zz{\mathbb{Z}}

\def\CCC{\mathcal{C}}
\def\DDD{\mathcal{D}}

\def\FFF{\mathcal{F}}
\def\GGG{\mathcal{G}}
\def\HHH{\mathcal{H}}

\def\LLL{\mathcal{L}}

\def\NNN{\mathcal{N}}

\def\PPP{\mathcal{P}}

\def\ZZZ{\mathcal{Z}}

\renewcommand{\le}{\leqslant}
\renewcommand{\ge}{\geqslant}

\numberwithin{equation}{section}

\numberwithin{equation}{section}
\theoremstyle{plain}

\newcommand{\bp}{\begin{proof}}
\newcommand{\ep}{\end{proof}}
\newcommand{\bl}{\begin{lemma}}
\newcommand{\el}{\end{lemma}}
\newcommand{\bt}{\begin{theorem}}
\newcommand{\et}{\end{theorem}}
\newcommand{\bd}{\begin{definition}}
\newcommand{\ed}{\end{definition}}
\newcommand{\ba}{\begin{arrow}}
\newcommand{\ea}{\end{arrow}}

\begin{document}


%
%
%
%

\begin{abstract}
	We apply the methods developed earlier in  Skriganov, 
	\textit{St.Petersburg Math. J.}, {\bf 6}(3), (1995), 635--664, and
	\textit{Inventiones Math}, 
	132(1), 1998, 1--72, to estimate the following sums arising in the context of 
 integer point counting in polyhedra 
	\begin{align*}
		\mathscr{S}(\,{\bf{\theta}},\varphi,\bfu, T)= 
		\,\sum_{m_d\ne0} 
		\frac{e^{2i\pi m_du_d}}{m_d}
		\prod_{1\le l\le d-1}\,\,
		\sum_{m_l\in\Zz}
		\frac{e^{2i\pi m_lu_l}}{\theta_{l}m_d -m_l}\,\, \varphi \,(T^{-1}\bfm)\, 
		,  
		\label{eq6.000}
	\end{align*}	
	where $\theta_1,\dots,\theta_{d-1}$ are real irrational numbers,
	$\bfm=(m_1,\dots,m_d)\in\Zz^d$, and $\varphi\,(\bfx), \bfx\in\Rr^d,$ is
	a rapidly decreasing function as $\bfx\to\infty$.
	
We show that such sums can be estimated in terms of 
\textit{multiplicative Diophantine approximation} to 
$\theta_1,\dots,\theta_{d-1}$. In particular, 
we show that	
\begin{equation*}
	\mathscr{S}(\theta,\varphi ,T)=	\sup\nolimits_{\bfu\in\Rr^d} 
	\big|\mathscr{S}(\,{\bf{\theta}},\varphi, 
	\bfu, T)\big|=	O(T^{\epsilon})\,,\quad T\to\infty\,,
\end{equation*}		
with any $\epsilon>0$, if $1,\theta_1,\dots,\theta_{d-1}$ are real algebraic 
numbers linear independent over the field of rationals $\Qq$.

\end{abstract}

\maketitle

\thispagestyle{empty}

\section*{Contents}

\noindent{} {1. Introduction and Main results} 

\noindent{} {2. Facts from Geometry of numbers}



\noindent{} {3. Dyadic minima of a lattice}
	
\noindent{} {4. Dyadic decompositions of lattice sums}




\noindent{} {5. Fourier integrals and lattice sums}

\noindent{} {6. Proof of Main bounds for lattice sums}

References
\enlargethispage{4\baselineskip}

\medskip

\section{Introduction and main results}
\label{sec1}

\label{sec6}
In this paper, we will study sums of the form 
\begin{align}
	\mathscr{S}(\,{\bf{\theta}},\varphi, T,\bfu)= 
	\,\sum_{m_d\ne0} 
	\frac{e^{2i\pi m_du_d}}{m_d}
	\prod_{1\le l\le d-1}\,\,
	\sum_{m_l\in\Zz}
	\frac{e^{2i\pi m_lu_l}}{\theta_{l}m_d -m_l}\,\, \varphi (T^{-1}\bfm)\, ,  
	\label{eq6.00}
\end{align}
where $\theta_1,\dots,\theta_{d-1}$ are real irrational numbers. 
This series is well defined if the small 
denominators  are compensated by a fast decay of $\varphi 
(\bfm)$ as $\bfm\to\infty$.

In the two-dimensional case, such sums are well studied and their estimates 
are described in terms of Diophantine approximations to an irrational 
number $\theta_1$, see, for example, \cite[Chap. 2, Sec. 3]{4****} and 
related historical remarks given there. In the multidimensional case, the situation 
is much more complicated, and estimates of such sums involve specific  simultaneous 
Diophantine approximations of irrational numbers $\theta_1,\dots,\theta_{d-1}$.
\\

\textbf{Definition 1.1.}
For $\kappa>0$, a set of numbers 
$\theta=(\theta_1,\dots,\theta_{d-1})\in\Rr^{d-1}$ is said to be
$\kappa$-\textit{multiplicatively approximable}, if the inequality
\begin{equation}
	m^{1+\kappa}\,\prod\nolimits_{\,j\in [d-1]}\,\, \langle 
	\theta_j\,\, m\rangle \ge c\,(\kappa) \,,
	\label{eq6.02*}
\end{equation}
is satisfied for all $m\in \Zz_{>0}$ with a constant $c\,(\kappa)>0$.
Here  $\langle x\rangle$ denotes the distance between $x\in\Rr$ and the set of 
integers $\Zz$, and we write 
$[n]=\{1,\dots,n\}$ for $n\in\Zz_{\ge 0}$.

If \eqref{eq6.02*} is satisfied with arbitrary small $\kappa >0$,
the set  $(\theta_1,\dots,\theta_{d-1})$ is said to be \textit{badly 
	multiplicatively approximable}.

\textbf{Remarks.}
\textit{(i)} It follows from the metric theory of Diophantine 
approximation that almost all 
$\theta = (\theta_1,\dots,\theta_{d-1})\in\Rr^{d-1}$ 
are badly multiplicatively approximable; more precisely, we have the inequality
\begin{equation*}
m\,(\log m)^{d-1+\epsilon}\,\prod\nolimits_{\,j\in [d-1]}\,\, \langle 
	\theta_j\,\, m\rangle \ge c_{\epsilon} 
\end{equation*}
with any $\epsilon >0$ and a constant $c_{\epsilon} >0$, 
see \cite[Chap. 1, Sec. 8]{25} and references therein.

\textit{(ii)}
It follows from  Schmidt's theorem on Diophantine approximation that a set
$(\theta_1,\dots,\theta_{d-1})$ is
badly multiplicatively  approximable if $1,\theta_1,\dots,\theta_{d-1}$ 
are real algebraic numbers linear independent over $\Qq$, see \cite[Chap. VI, 
Thm.1B]{19}.  

\textit{(iii)}.
For $d=2$ there are irrationalities $\theta_1$ (quadratic irrationalities, for 
example) such that \eqref{eq6.02*} 
is satisfied with $\kappa=0$.
At the same time, a famous conjecture of Littlewood states that  
there are 
no real numbers $\theta_1,\dots,\theta_{d-1}$ for $d\ge3$ to satisfy 
\eqref{eq6.02*} with 
$\kappa =0$, see \cite[Chap.V, Sec. 10.3] {4***},  \cite[Chap.II, Sec. 4] {19}, and 
\cite [Sec. 30.3]{26}.
\\

Recall that the space $\frak{S}(\Rr^d)$ of \textit{rapidly decreasing} 
functions consists of $\CCC^{\infty}$ functions $\varphi (\bfx),\,\bfx\in\Rr^d,$
which satisfy
\begin{equation*}
	\big|\,\frac{\partial^{b_1+\cdots+b_d}}{\partial x^{b_1}
		\dots\partial x^{b_d}}\,\,\varphi(\bfx)\,\big| < C_{\varphi,\bfb,a}\, 
		(\,1+|\bfx|_{\infty})^{-A},
		\quad |\bfx|_{\infty}=\max\nolimits_{j\in[d]}\,|x_j|\,,
	\label{eq6.02**}	
\end{equation*}
for all multi-indexes $\bfb=(b_1,\dots,b_d)\in\Zz_{\ge0}^d$
and with arbitrary large $A>0$, see \cite[Chap. 1, Sec. 3]{27}. 

In this paper, we will prove the following.

\begin{theorem}
	Let a set of numbers 
	$\theta = (\theta_1,\dots,\theta_{d-1})$  be
	$\kappa$-multiplicatively approximable with $\kappa>0$ and 
	$\varphi\in\frak{S}(\Rr^d)$. Then sum \eqref{eq6.00}
	satisfies the bound
	\begin{equation}
		\mathscr{S}(\theta,\varphi,T)=\sup\nolimits_{\bfu\in\Rr^d} 
		\big|\mathscr{S}(\,\theta,\varphi,T,\bfu,\,)\big|=		
		O\,\big(T^{\kappa}\,(\log T)^{d-1}\, \big)\, .
		\label{eq6.01*}
	\end{equation} 
	
	If a set of numbers $\theta = (\theta_1,\dots,\theta_{d-1})$ is badly 
	multiplicatively approximable, then \eqref{eq6.01*}
	takes the form
	\begin{equation}
		\mathscr{S}(\theta,\varphi,T)=		
		O_{\epsilon}\,\big(T^{\epsilon}\, \big)
		\label{eq6.02}
	\end{equation} 
	with any $\epsilon>0$.
	In particular, \eqref{eq6.02} is satisfied, if 
	$1,\theta_1,\dots,\theta_{d-1}$ are real algebraic numbers linear 
	independent over $\Qq$.
\end{theorem} 

\textbf{Remarks.}
\textit{(i)} With the hypothesis of Theorem~1.1, 
the series \eqref{eq6.00} converges absolutely and uniformly for $\bfu\in\Rr^d$.
Indeed, it follows from \eqref{eq6.02*} that
\begin{align*}
	|m_j\prod\nolimits_{l\in [d], l\ne j}(\theta_{l,j}m_j -m_l)|
	\ge |m_j|\prod\nolimits_{l\in [d]\setminus {j}}\,\, \langle 
	\theta_{j,l}\,\, m\rangle >c_{\kappa} \, |m_j|^{-\kappa}. 
\end{align*}
Therefore,
\begin{align}
	|\mathscr{S}(\,&\bf{\theta},\,\omega, T,\bfu)|
	\notag
	\\
	&\le \,\sum_{m_d\ne0}\, 
	\sum_{m_1,\dots,m_{d-1}\in\Zz}
	\frac{1}{|m_d\,(\theta_{1}m_d -m_l),\dots,(\theta_{d-1}m_d -m_{d-1})|}\,\, 
	|\varphi 
	(T^{-1}\bfm)|
	\notag
	\\
	&\ll\, \sum_{m_d>0}\, \,
	m_d^{\kappa}
	\,\, 
	\sum_{m_1,\dots,m_{d-1}\in\Zz}
	|\varphi 
	(T^{-1}\bfm)| =O_{\theta,\varphi}(\,T^{d+\kappa}\,)\,.
	\label{eq6.00a}
\end{align}
 
\textit{(ii)} The bound \eqref{eq6.01*} is sharp up to a 
logarithmic multiplier if the lower bound \eqref{eq6.02*} with some $\kappa>0$ is 
achieved on a subsequence of 
$m\to\infty$. Indeed, in this case, the corresponding Fourier coefficients 
in \eqref{eq6.00} are of order $\gg m^{\kappa}\,|\varphi (T^{-1}\bfm)|$ and so  
$\mathscr{S}(\theta,\omega,T)\gg T^{\kappa}$ on a subsequence of $T\approx 
m\to\infty$.

\textit{(iii)} 
It follows from results of Spencer \cite{24} that for almost all 
$\theta = (\theta_1,\dots,\theta_{d-1})\in\Rr^{d-1}$ the sum \eqref{eq6.00}  
satisfies the bound
$$\mathscr{S}(\theta,\varphi,T)=O_{\epsilon}\big((\,\log T\,)^{d+\epsilon}\big)$$ 
with any $\epsilon >0$.
In this paper, we do not consider metric results of this type, since our main 
interest is focused on sub-polynomial estimates for specific 
$\theta_1,\dots,\theta_{d-1}\in\Rr^{d-1}$, such as algebraic numbers.
\\

The results given  in Theorem~1.1 are stated in the form needed for applications to 
the problem of integer point counting. Such applications will be considered in the 
forthcoming paper   \cite{22*}. 
We will  derive Theorem 1.1 from Theorem 1.2 given below, where \eqref{eq6.00} is 
written as a lattice sum.

Let us introduce the following notation. For 
$\bfx=(x_1,\dots,x_d)\in\Rr^d$, 
we write $\bfx=(X,x_d)$, where $X=(x_1,\dots,x_{d-1})\in\Rr^{d-1},\, x_d\in\Rr$,
and $\bfx\centerdot\bfy=X\centerdot Y+x_d\,y_d$, where 
$X\centerdot Y=\sum\nolimits_{j\in[d-1]}\,x_j\,y_j$.

For a set of numbers $\theta=(\theta_1,\dots,\theta_{d-1})\in\Rr^{d-1}$, we 
define a lattice $\Lambda_{\theta}\subset\Rr^d$ by setting
\begin{align*}
\Lambda_{\theta}=&\{\,\bfx: x_j=\theta_j 
m_d-m_j,\,j\in[d-1],\,x_d=m_d,\,\bfm\in\Zz^d\,\}\notag
\\
=&\{\,\bfx=(X,x_d):X\in\Zz^{d-1}+x_d\,\theta,\,x_d\in\Zz\,\}\notag
\\
=&\bigsqcup\nolimits_{\,x_d\in\Zz^d\,}
\{\,\bfx=(X,x_d):X\in\Zz^{d-1}+x_d\,\theta\,\}\, .	  
\end{align*} 
In other words, $\Lambda_{\theta}$ is a disjoint union of shifted lattices
 $\Zz^{d-1}+x_d\,\theta$ located on the hyperplanes $\{\bfx:x_d\in\Zz\}$.
Notice that the $d-1$ dimensional cube $K=[-\frac12,\frac12)^{d-1}$ is 
a fundamental set for the lattice $\Zz^{d-1},\,K=\Rr^{d-1}/\Zz^{d-1}$.
Therefore, $K$ contains exactly one point of the shifted lattice 
$\Zz^{d-1}+x_d\,\theta$ for any $x_d\in\Zz$. In fact, such a point 
 is contained
in the open cube $K^{int}=(-\frac12,\frac12)^{d-1}$ if 
$\theta_1,\dots,\theta_{d-1}$ are irrational.

The dual lattice $\Lambda_{\theta}^{\perp}$ to $\Lambda_{\theta}$ is given by
\begin{align*}
	\Lambda_{\theta}^{\perp}=&\{\,\bfy: y_j=n_j,\, n_j\in\Zz,\,j\in[d-1],\,
	y_d=-\sum\nolimits_{j\in[d-1]}\,\theta_j\,n_j\,\}\notag
	\\
	=&\{\,\bfy=(Y,y_d):Y\in\Zz^{d-1},\, y_d=-\,\theta\centerdot Y\,\}\, .
\end{align*} 
Both lattices $\Lambda_{\theta}$ and $\Lambda_{\theta}^{\perp}$ are unimodular.
 Define a subset $\Lambda_{\theta}^{\natural}\subset\Lambda_{\theta}$ by setting
\begin{align*}
	\Lambda_{\theta}^{\natural}=\bigsqcup\nolimits_{x_d\in\Zz\setminus \{0\}}\,
	\{\,\bfx=(X,x_d): X\in\Zz^{d-1}+x_d\,\theta\,\}	\, .  
\end{align*} 

Consider the lattice sum
\begin{align}
	S(\,\Lambda_{\theta},\phi,\bfu, T\,)=
	\,\sum\nolimits_{\bfx\in\Lambda_{\theta}^{\natural}}\,\,\, 
	\frac{e^{2i\pi\,\bfu\centerdot\bfx}}{x_1\dots x_d}\,\,\,
	\phi (T^{-1}\bfx)\,,  
	\label{eq6.06*}
\end{align}	
One can easily check that the sums \eqref{eq6.00} and \eqref{eq6.06*} coincide: 
\begin{align*}
	S(\,\Lambda_{\theta},\phi,\bfv, T\,)=\mathscr{S}(\,\theta,\varphi,\bfu, T\,)\, ,
\end{align*}	
provided that $v_l=-u_l,\, l\in [d-1],\,
v_d=u_d+\sum\nolimits_{l\in [d-1]}\theta_l\, u_l,$ and
\begin{equation*}
	\phi(x_1,\dots,x_d)=
	\varphi(\theta_1\,x_d-x_1,\dots,\theta_{d-1}\,x_d-x_{d-1},\,x_d\,)\,.
	\label{eq6.08*}
\end{equation*}
This shows that Theorem~1.1 is a paraphrase of the following.

\begin{theorem}[Main bounds for lattice sums]
	Let a set of numbers 
	$\theta=(\theta_1,\dots,\- \theta_{d-1})\in\Rr^{d-1}$  
	be $\kappa$-multiplicatively approximable with $\kappa>0$ and 
	$\phi\in\frak{S}(\Rr^d)$.
 Then the lattice sum \eqref{eq6.06*} satisfies the bound
	\begin{equation}
	S(\,\Lambda_{\theta},\phi,T\,)=
		\sup\nolimits_{\bfu\in\Rr^d} 
	\big|S(\,\Lambda_{\theta},\phi,\bfu,T\,)\big|=	
		O\,\big(\,T^{\kappa}\,(\log T)^{d-1}\, \big)\, .
		\label{eq6.07}
	\end{equation} 
	
	If a set of numbers $\theta = (\theta_1,\dots,\theta_{d-1})$ is badly 
	multiplicatively approximable,  \eqref{eq6.07}
	takes the form
	\begin{equation}
		S\,(\,\Lambda_{\theta},\varphi,T\,)=		
		O_{\epsilon}\,\big(T^{\epsilon}\, \big)
		\label{eq6.02a}
	\end{equation} 
	with any $\epsilon>0$.
	In particular, \eqref{eq6.02a} is satisfied, if 
	$1,\theta_1,\dots,\theta_{d-1}$ are real algebraic numbers linear 
	independent over $\Qq$.
 \end{theorem}

The main ideas of the proof of Theorem 1.2 can be outlined as follows.

The set of  unimodular lattices $\Gamma\subset\Rr^d$ can be identified with
the homogeneous space $\LLL_d=SL_d(\Rr)/SL_d(\Zz)$.
A fundamental fact of the modern theory of Diophantine approximation is that many 
 problems in this field can be formulated, investigated and solved 
 using geometry of the space of lattices $\LLL_d$ and methods from the theory of 
 flows on homogeneous spaces.
 We refer to \cite{16, 16a, 26} for 
an overview of such an approach to Diophantine analysis. 

This approach has been used in our papers \cite{21, 22}
to estimate lattice sums arising in the lattice point problem for polyhedra.
It has been shown that such sums for lattices $\Gamma\in\LLL_d$ can be expressed 
in terms of special sums over the orbits $\{D\,\Gamma^{\perp}:D\in\mathscr{D}\}$ of 
dual lattices $\Gamma^{\perp}$ under the action of the subgroup  of diagonal 
matrices
\begin{equation*}
\mathscr{D}	=\{\,D(\bfa)=
	\text{diag}\, (2^{a_1},\dots,2^{a_d}):
	a_1+\cdots + a_d=0\,\}\, .
\end{equation*}
The key points in the construction of sums over orbits are dyadic  
minima 
of lattices and dyadic decompositions of lattice sums.
In this paper, 
the approach  given in \cite{21, 22}  
will be
simplified and  adapted to lattices $\Lambda_{\theta}$. 
The proof of Theorem 1.2 will be given as a sequence of simple facts and results 
from Geometry of numbers and Fourier analysis.


\section{Facts from  geometry of numbers}
\label{sec2}

\textit{2.1. Successive minima of a lattice.}
Let $\Gamma\in\LLL_d$ be an unimodular lattice and 
$$0<\lambda_1(\Gamma)\le\dots\le\lambda_d(\Gamma) $$
denote successive minima of $\Gamma$ with respect to the norm 
$|\cdot|_{\infty}$
Recall that 
$\lambda_j(\Gamma)$ is the minimum value of $t$ such that
the cube $K_{t}=\{\bfx:|\bfx|_{\infty}\le t\}$ contains at least
$j$ linear independent points of $\Gamma$. In particular, 
\begin{equation}
\lambda_1(\Gamma)=\min\nolimits_{\bfx\in\Gamma\setminus\{0\}}\,|\bfx|_{\infty}
\label{eq6.08}
\end{equation}

A theorem of Minkowski  states that successive minima satisfy
\begin{equation*}
	(d!)^{-1}\le\,\lambda_1(\Gamma)\dots\lambda_d(\Gamma)\,\le 1\, ,
\end{equation*}
so that $\lambda_1(\Gamma)$ is bounded from above and $\lambda_d(\Gamma)$ from 
below by constants independent of $\Gamma$, see \cite[Chap.VIII, Sec.4, 
Thm V]{4***} and \cite[Chap.2, Sec.9, Thms~1~and~2]{12}.

A theorem of Mahler states that successive minima of dual lattices $\Gamma$
and $\Gamma^{\perp}$ are related by
\begin{equation*}
d^{-1}\le\, \lambda_j(\Gamma)\,\,\lambda_{d-j+1}(\Gamma^{\perp})\,\le (\,d!\,)^2 ,	
	\end{equation*}
	see \cite[Chap.VIII, Sec.5, Thm VI]{4***} 
	and \cite[Chap.2, Sec.14, Thm 5]{12}.
\\	
	

\textit{2.2. Elementary lattice sums.}	We write 
\begin{equation*}
N(t,\Gamma,\bfu)=\#\,\{(K_t+\bfu)\cap\Gamma\}=
\sum\nolimits_{\bfx\in\Gamma}\chi (t,\bfx-\bfu)	
	\end{equation*}	
	 for the number of lattice points in the shifted cube $K_t+\bfu$; here 
	 $\chi(t,\cdot)$ is the indicator function of $K_t$. We also put
\begin{equation*}
	N(t,\Gamma)=\sup\nolimits_{\bfu\in\Rr^d} N(\Gamma,t,\bfu).	
\end{equation*}

	For $A>d$ and $\nu>0$, we define  
	an \textit{elementary lattice sum} by setting
	\begin{equation*}
	f_A(\nu,\Gamma,\bfu)=\sum\nolimits_{\bfx\in\Gamma}
	\, (\,1+\nu\,|\bfx-\bfu|_{\infty})^{-A}\,.
		\end{equation*}
	It is clear that this series converges and $f_A(\Gamma,\nu,\bfu)$ is a positive 
	periodic function of 
	$\bfu\in\Rr^d$. We  put	
\begin{equation*}
	f_A(\nu,\Gamma)=\sup\nolimits_{\bfu\in\Rr^d}f_A(\Gamma,\nu,\bfu)\,.
\end{equation*}		

\begin{proposition}
	(i) We have
	\begin{align*}
N(t,\Gamma)\le
c\,\big(\,t^d +\lambda_d\,(\Gamma)^d\,\big)
	\le C\,\big(\,t^d +\lambda_1\,(\Gamma^{\perp})^{-d}\,\big)\, 
		\end{align*}
with  constants independent of $\Gamma$.
		
		(ii) We have
		\begin{align*}
		f_A(\nu,\Gamma)
		\le c_A\,\big(\,\nu^{-d} +\lambda_d\,(\Gamma)^d\,\big)
		\le C_A\,\big(\,\nu^{-d} +\lambda_1\,(\Gamma^{\perp})^{-d}\,\big)\,.
	\end{align*}	
	with  constants  depending only on $A>d$.
	\end{proposition}
	\begin{proof}
\textit{(i)} It follows from the known results of geometry of numbers that a 
lattice $\Gamma\in\LLL_d$ has a basis $\bfb_1,\dots,\bfb_d$ such that
$|\bfb_j|_{\infty}\le\frac{d}{2}\lambda_d(\Gamma), \,\, j\in [d]$, 
see \cite[Chap. V, Lemma 8]{4***}. Consider a parallelepiped $\PPP$ spanned on 
this basis,
$$\PPP=\{\,\bfy: \bfy=t_1\,\bfb_1+\cdots +t_d\,\bfb_d,\,\,0\le t_j<1,\,\, 
j\in[d]\,\}.$$
The volume and diameter of $\PPP$ satisfy
$\text{vol}\,\PPP=1$ and 
$\text{diam}\,\PPP\le\frac{d^2}{2}\lambda_d\,(\Gamma)$. 
$\PPP=\Rr^d/\Gamma$ is a fundamental set of $\Gamma$, so that each shifted set 
$\PPP+\bfx$ 
contains exactly one lattice point, and
we have the partition
$
\Rr^d=\bigsqcup\nolimits_{\bfx\in\Gamma}\,
\{\,\PPP +\bfx\,\}	\, 
$. 
Therefore,
\begin{align*}
N(t,\Gamma,\bfu)&\le\#\{\bfx\in\Gamma: (\PPP+\bfx)\cap (K_t+\bfu)\ne\emptyset\}
\notag
\\
&\le \text{vol}\, K_{t+ \text{diam}\PPP}=2^d\,\big( 
t+\text{diam}\,\PPP\big)^d
\notag
\\
&\le 2^d\,\Big( t+\frac{d^2}{2}\lambda_d(\Gamma)\Big)^d
\notag
\\
&\le c\,\big(\,t^d +\lambda_d\,(\Gamma)^d\,\big)
\le C\,\big(\,t^d +\lambda_1\,(\Gamma^{\perp})^{-d}\,\big)\,,
\end{align*}
where Mahler's theorem 
has been used in the latest line.

\textit{(ii)}
The sum $f_A(\nu,\Gamma,\bfu)$ can be written as a Stieltjes integral:
\begin{align*}
	f_A(\nu,\Gamma,\bfu)&=\int\nolimits_0^{\infty} f_A(\nu\,t)\,d N(t,\Gamma,\bfu)
	\notag
	\\
	&=f_A(0)N(0,\Gamma,\bfu)-\int\nolimits_0^{\infty} \nu 
	f'_A\,(\nu\,t)\, 
	N(t,\Gamma,\bfu)\, dt
	\,.
\end{align*}
The part (i) implies  
\begin{align*}
	f_A(\Gamma,\nu,\bfu)\ll f_A(0)+\nu^{-d}\,\int\nolimits_0^{\infty} 
	t^d\,|f'_A(t)|\,dt\, 
	+\lambda_d(\Gamma)^d\,\int\nolimits_0^{\infty} |f'_A(t)|\,dt
\end{align*}
and the required bound follows.
\end{proof} 
	
\section{Dyadic minima of a lattice}
\label{sec3}

\textbf{Definition 3.1.}
For  $L\in\Zz^{d-1}_{\ge0}$, 
consider the following cylindrical subsets
\begin{align*}
\CCC(L)=\{\,\bfx\in\Rr^d: |x_j|\le 2^{-l_j-1},\,j\in[d-1],\, x_d\in\Rr\,\}\,.
\end{align*}
The \textit{dyadic minima} of a lattice 
$\Lambda_{\theta}\subset\Rr^d$	are positive integers 
$\mu(L),\,L\in\Zz^{d-1}_{\ge0},$ defined by
\begin{equation}
	\mu(L)=\min\,\{\, |x_d|: \bfx\in\Lambda_{\theta}^{\natural}\cap\CCC(L)\,\}\, ,
	\label{eq6.014*}	
\end{equation}	

A few immediate consequences of the definition can be mentioned.

\textit{(i)} If $L=0$, then $\mu(0)=1$.

\textit{(ii)} If at least one of $\theta_1,\dots,\theta_{d-1}$ is rational, then
$\mu(L)\le m$, for all $L\in\Zz_{\ge0}^{d-1}$, where $m>0$ is the  least
integer such that the point 
$(0,\dots,0,m) \in\Lambda^{\natural}_{\theta}$.

\textit{(iii)} If all $\theta_1,\dots,\theta_{d-1}$ are irrational, then 
$\mu(L)\to\infty$ as $L\to\infty$.

\textit{(iv)} If $\mu(L)$ and $\mu(L')$ are two dyadic minima and 
$l_j\ge l'_j,\,j\in[d-1]$, then $\mu(L)\ge m(L')$. 

Instead of $L\in\Zz^{d-1}_{\ge0}$ in \eqref{eq6.014*} one could consider 
$L\in\Zz^{d-1}$. However, this does not lead to any generalization of the 
definition, because it is easy to prove the relation $\mu(L)=\mu(L^{+})$, where
$L^{+}=(l_1^{+},\dots,l_{d-1}^{+})$, and 
$l^{+}=\max\{l,0\},\,l\in\Zz$.

Here is a result of this type, which will be needed later.
Let us introduce additional notations. 
For a subset $J\subseteq [d-1]$, we write
$J'=[d-1]\setminus J$,  $|J|=\#\{J\}$.  
Define subsets $\ZZZ_J\subset\Zz^{d-1}_{\ge 0}$ by
\begin{equation*}
	\ZZZ_J=\{\,L\in\Zz_{\ge 0}^{d-1}: l_j\in\Zz_{>0},\,\,\text{if}\,\, j\in J
	\,\,\text{and}\,\, l_j=0,\,\,\text{if}\,\, j\in J'\,\}\,.
\end{equation*}
Notice that these subsets form a partition of $\Zz^{d-1}_{\ge 0}$,
\begin{align*}
	\Zz^{d-1}_{\ge 0}=\,\bigsqcup\nolimits_{J\subseteq [d-1]}\,\ZZZ_J\, .  
\end{align*} 
For $L\in\ZZZ_J$, consider the following cylindrical subsets
\begin{align}
	\CCC_J(L)=\{\bfx\in\Rr^d: 
	&|x_j|\le2^{-l_j-1},\,\,\text{if}\,\, j\in J,\,\,
	\notag
	\\
	\,\,&x_j\in\Rr,\,\,\text{if}\,\,j\in J',\,\,\text{and}\,\, 
	x_d\in\Rr\}\,.	
	5\label{eq6.0015}
\end{align}	
and define the minima
\begin{equation*}
	\mu_J\,(L)=\min\,\{\, |x_d|: 
	\bfx\in\Lambda_{\theta}^{\natural}\cap\CCC_J(L)\,\}\,.
\end{equation*}

It is clear that, for $J=[d-1],\,\, \CCC_{[d-1]}(L)=\CCC(L)$ and 
$\mu_{[d-1]}\,(L)=\mu(L)$.

\begin{lemma}
	Let $L\in\ZZZ_J$, then $\mu_J\,(L)=\mu(L)$.
\end{lemma}
\begin{proof}

	Since $\mu_{[d-1]}\,(L)=\mu(L)$, we may assume that $J\subset [d-1]$ is a 
	proper 
	subset. 
	
	Let the 
	minimum $\mu_J\,(L)$ is achieved at a point 
	$\bfx=(X,\mu_J(L))\in\Lambda_{\theta}^{\natural}\cap\CCC_J(L)$.
	Consider the lattice points
	$(X+Z,\mu_J(L)),\,\, Z\in\Zz^{|J'|}$.
	Exactly one of these points belongs to $\CCC(L)$,
	since the cube $[-\frac12,\frac12)^{|J'|}$ is a fundamental set of $\Zz^{|J'|}$.
	Therefore, $\mu_J\,(L)\ge\mu(L)$.
	
	On the other hand, let the minimum $\mu\,(L)$ is achieved at a point 
	$\bfx=(X,\mu(L))\in\Lambda_{\theta}^{\natural}\cap\CCC(L)$.
	Consider the lattice points $(X+Z,\,\mu(L)),\,\, Z\in\Zz^{|J'|}\setminus\{0\}$.
	All these points belong to $\CCC_J(L)$. Therefore,  $\mu_J\,(L)\le\mu(L)$.
	This proves the lemma.
\end{proof}

A vector $\bfy\in\Rr^d_{>0}$ can be written as
		\begin{equation*}
			\bfy=\nu(\bfy)\,\,(2^{-\alpha_1},\dots,2^{-\alpha_d})
			\quad\text{with}\quad\alpha_1+\cdots +\alpha_d=0\, 
			\end{equation*}
and $\nu(\bfy)=(y_1\dots y_d)^{1/d}$.	
	The diagonal matrix 
	$D_{\bfy}=\text{diag}\,(2^{\,\alpha_1},\dots,2^{\,\alpha_d})\in\DDD$.
satisfies $D_{\bfy}\,\bfy=\nu(\bfy)\,(1,\dots,1)$.

We will call $\nu(\bfy)$ the \textit{hyperbolic norm} and  $D_{\bfy}$  the 
\textit{equalizer} of  $\bfy$.	
	
For each dyadic minimum $\mu(L)$, we associate a vector $\bfm(L)\in\Rr^d_{>0}$ 
defined by	
\begin{equation}
\bfm(L)=(\,2^{-l_1-1},\dots,2^{-l_{d-1}-1},\mu(L)\,)\,.
\label{eq6.0012}
\end{equation} 
The equalizer  
$D_{\bfm(L)}=\text{diag}\,(2^{\,\alpha_1},\dots,2^{\,\alpha_d})$
of $\bfm(L)$,
is determined by  the equations
\begin{equation}
\nu\,\,2^{-\alpha_j}=2^{-l_j-1},\, j\in[d-1], \quad
\nu\,\,2^{-\alpha_d}=\mu(L)\,,
\label{eq6.00012}
\end{equation}
where $\nu=\nu(\bfm(L))$ is defined by
\begin{equation}
(\nu)^d=(\nu(\bfm(L)))^d=\mu(L)\,\prod\nolimits_{j\in [d-1]}\,2^{-j_j-1}=
2^{1-d}\,\mu(L)\,2^{-|L|_1}\,,
\label{eq6.000012}
\end{equation}
where $|L|_1=l_1+\cdots+l_{d-1}$.
\\

Dyadic minima and multiple Diophantine approximations are intimately related. 
			\begin{proposition}
			 Let a set of numbers 
		$\theta = (\theta_1,\dots,\theta_{d-1})\in\Rr^{d-1}$  be
		$\kappa$-multiplicatively approximable.  Then	
		\begin{align}
		\mu(L)> \,c_1(\kappa)\,\, 2^{\frac{1}{\kappa +1}\, |L|_1}
		\label{eq6.012**}
	\end{align}
	and
	\begin{align}
		\lambda_1(\, D_{\bfm(L)}\,\Lambda_{\theta}\,)^d=\,\,
		\nu(\,\bfm(L)\,)^d> c_2(\kappa)\,\, 2^{-\frac{\kappa}{\kappa +1}\, |L|_1}
		\label{eq6.012*}	
			\end{align}
			with  positive  constants $c_1(\kappa)$ and $c_2(\kappa)$ depending 
			only on the constant $c(\kappa)$ in 
			\eqref{eq6.02*}. 
				\end{proposition}	
				\begin{proof}
				Let $\bfx=(x_1,\dots,x_d)\in\Lambda_{\theta}^{\natural}\cap\CCC(L)$	
denote the point where the minimum \eqref{eq6.014*} is achieved.	Then, 			
$|x_j(l)|< 2^{-l_j-1}\le 1/2,\,j\in[d-1]$, and we can write	
$$|x_j|=\langle x_j(l) \rangle =\langle\, 
\theta_j\, \mu(L)\, \rangle <2^{-l_j-1},\quad j\in[d-1]\,.$$
Therefore, in view of \eqref{eq6.02*},  we have				
		\begin{align*}
			\mu(L)^{1+\kappa}\,\prod\nolimits_{j\in [d-1]}2^{-l_j-1}>
			\mu(L)^{1+\kappa}\,
			\prod\nolimits_{j\in [d-1]}\,\, \langle \, 
			\theta_{l}\,\, \mu(L)\,\rangle \ge c\,(\kappa)\, . 
		\end{align*}
	Hence, 
	$\mu(L)>c_1(\kappa)\,2^{\frac{1}{\kappa 
	+1}\, |L|_1}$
with $c_1(\kappa)=\big(c\,(\kappa)\,2^{d-1}\big)^{\frac{1}{1+\kappa}}$. 			
		This proves \eqref{eq6.012**}. 
		
		Similarly,
		\begin{align*}
		(\nu&(\bfm(L)))^d =\mu(L)\,\prod\nolimits_{j\in [d-1]}2^{-l_j-1}
		\notag
		\\
		&=\big(\mu(L)^{1+\kappa}\,
			\prod\nolimits_{j\in [d-1]}\,2^{-1-l_j}\big)^{\frac{1}{1+\kappa}}\,
		\big(\prod\nolimits_{j\in 
		[d-1]}\,2^{-1-l_j}\big)^{\frac{\kappa}{1+\kappa}}	
		\notag
		\\
			&\ge c_2(\kappa)\, 2^{-\frac{\kappa}{1+\kappa}\,|L|_1}\, . 
		\end{align*}	
		with
	$c_2(\kappa)=\big(c\,(\kappa)\,2^{-(d-1)\kappa}\big)^{\frac{1}{1+\kappa}}$.
		This proves the right inequality in \eqref{eq6.012*}.
		
		To prove the left equality in \eqref{eq6.012*}, we consider 
		the rectangular box 
		$$\Pi_L=\{\bfy: |y_j|\le 2^{-l_j-1},\, j\in[d-1],\,\,|y_d|\le \mu(L)\,.\}$$
	The box contains a lattice point $\bfx$ on its boundary and does not contain 
	points of $\Lambda_{\theta}\setminus \{0\}$	in its interior.
The equalizer $D_{\bfm(L)}$ transforms 
$\Pi_L$ to the cube 
$$[\,-\nu,\, \nu\,]^d=\{\bfy: |\bfy|_{\infty}\le\nu\},\quad
\nu=\nu(\bfm(L))\,.$$ 
This cube contains the point $D_{\bfm(L)}\bfx$ of the lattice 
$D_{\bfm(L)}\Lambda_{\theta}$
on its boundary and does not contain 
points of  $D_{\bfm(L)}\Lambda_{\theta}\setminus \{0\}$	in  its 
interior. This, together with \eqref{eq6.08}, proves the left equality in 
\eqref{eq6.012*}.			
			\end{proof}

\section{Dyadic decompositions of lattice sums.}
\label{sec4}

The dyadic decomposition of lattice sums is given below in Proposition 4.1. To 
formulate this result, we need special partitions of unity on the space $\Rr^d$ and 
special Fourier integrals. 

\textit{4.1. Partitions of unity.} 
Consider a partition of the real axis defined by
\begin{equation*}
	\Rr=\big(\,\bigsqcup\nolimits_{l\ge0} V_l\, \big)\,\,\bigsqcup\, V_{\infty}\, ,
	\end{equation*}
where $V_l=[-l-1, -l),\,l\in\Zz_{\ge0},$ and $V_{\infty}=[0,\infty)$. We write 
$\chi_l(x)=\chi(V_l,x),\,l\in\Zz_{\ge0},$ and $\chi_{\infty}(x),\,x\in\Rr,$ for 
the indicator functions of these intervals. Obviously,
$\chi_l(x)=\chi_0(x+l),\,l\in\Zz_{\ge0}$ and 
\begin{equation}
1=\sum\nolimits_{l\in\Zz_{\ge0}\cup\{\infty\}}\chi_l(x)=
\sum\nolimits_{l\in\Zz_{\ge0}}\chi_(x+l)+\chi_{\infty}(x),\quad	x\in\Rr\, .
	\label{eq6.014}
\end{equation}

We will define smoothed ``indicator functions'' as follows.
Let $\omega(\cdot)$ be a nonnegative $\CCC^{\infty}$ function 
supported on the interval $[-\frac12,\frac12]$ which satisfies 
$\int\nolimits_{\Rr}\,\omega(x)\,dx=1$. Consider the convolutions
\begin{align*}
	&\chi_l*\omega(x)=\int\nolimits_{\Rr}\,\chi_l(x)\omega(x-y)\,dy\, ,\quad
	l\in\Zz_{\ge0}, 
	\\ 
	&\chi_{\infty}*\omega(x)=\int\nolimits_{\Rr}\,\chi_{\infty}(x)\omega(x-y)\,dy.
\end{align*}

It is clear that $\chi_l*\omega(x)=\chi_0*\omega(x+l),\,l\in\Zz_{\ge0}$.
Moreover,
$\chi_0*\omega(x)$ and $\chi_{\infty}*\omega(x)$ are 
$\CCC^{\infty}$ functions supported on $[-2,1]$, and $[-1, \infty)$, respectively,
and $\chi_{\infty}*\omega(x)=1$ for $x\ge 1$.
It follows from \eqref{eq6.014} that
\begin{align}
   1=\sum\nolimits_{l\in\Zz_{\ge0}}\chi_0*\omega(x+l)+\chi_{\infty}*\omega(x)
	,\quad	x\in\Rr.
	\label{eq6.016*}
\end{align}
Thus, $\chi_0*\omega(x+l),\,
l\in\Zz_{\ge0},$ together with $\chi_{\infty}*\omega(x)$ form a partition 
of unity on $\Rr~$.

Let us introduce the functions
\begin{align*}
	\xi_0(x)=\chi_0*\omega\,\big(\,\log2|x|\,\big),\quad
	\xi_{\infty}(x)=\chi_{\infty}*\omega\,\big(\,\log2|x|\,\big)\,,
	\quad x\in\Rr\setminus\{0\}\,.
\end{align*}
Here and in what follows $\log$ denotes the logarithm to the base 2.

It is clear that $\xi_0(x)$ and $\xi_{\infty}(x)$ are nonnegative even 
$\CCC^{\infty}$ functions supported on
$[-1,-2^{-3}]\cup[2^{-3},1]$ and 
$(-\infty, -2^{-2}]\cup [2^{-2},\infty)$, respectively, and  $\xi_{\infty}(x)=1$ 
for $|x|\ge1$.
It follows from \eqref{eq6.016*} that
\begin{equation}
 1=	\sum\nolimits_{l\in\Zz_{\ge0}}\xi_0 (2^{l+1} x) +\xi_{\infty}(2x),\quad	
	x\in\Rr\setminus\{0\}\, .
	\label{eq6.0016**}
\end{equation}
Thus, $\xi_0(2^{l+1} x),\,l\in\Zz_{\ge0},$ together with $\xi_{\infty}(2 x)$ form 
a partition of unity on $\Rr\setminus\{0\}$. 
Note that for each $x\in\Rr$, at most three terms in the sum 
 \eqref{eq6.0016**} do not vanish.

Now we will define the  multidimensional partitions of unity as follows. 
For $J\subseteq [d-1],\,J' = [d-1] \setminus J, \, L\in\ZZZ_J$, and 
$X=(x_1,\dots,x_{d-1})\in\Rr^{d-1}$,
we define the functions 
\begin{equation*}
\xi_J(L,X)=
\prod\nolimits_{j\in J}\,\xi_0 (\,2^{l_j +1}\,x_j\,)\, 
\prod\nolimits_{j\in J'}\,\xi_{\infty}(\,2x_j\,).
\end{equation*}
We consider $\xi_J(L,\bfx)=\xi_J(L,X)$ as a function of
$\bfx=(X,x_d)\in\Rr^d$ independent of the coordinate $x_d$. 
From the above definitions and formulas, it follows that the functions
$\xi_J (L,\bfx)$ are supported on cylindrical subsets of the form
\begin{equation*}
	\prod\nolimits_{j\in J} [-2^{-l_j -1},-2^{-l_j-4}]\cup [2^{-l_j -4},2^{-l_j-1}]
	\times	\prod\nolimits_{j\in J'} [-\infty,-2^{-3}]\cup [2^{-3},\infty]
	\times\Rr \,.
\end{equation*}
This implies the following inclusions for the supports of  
$\xi_J (L,\bfx)$
\begin{equation}
	\text{support}\,\,\xi_J (\,L,\cdot\,)\subset\CCC_J(L) \,,
	\label{eq6.0016**aaa}
\end{equation}
where $\CCC_J(L)$ are the cylindrical subsets defined in \eqref{eq6.0015}.

Writing  \eqref{eq6.0016**} for $x=x_j,\,j\in [d-1],$ and multiplying 
these relations, we obtain\
\begin{equation}
	1=\sum\nolimits_{J\in [d-1]} \sum\nolimits_{L\in\ZZZ_J}\xi_J (L,\bfx),
	\quad	
	\bfx\in\Rr_{\star}^d \, ,
	\label{eq6.0016**a}
\end{equation}
where 
$$\Rr_{\star}^d=\big(\,\Rr\setminus\{0\}\,\big)^{d-1}\times\Rr=
\Rr^d\,\setminus\, \bigcup\nolimits_{j\in[d-1]}\,\{\bfx:x_j=0\}.$$
In other words, $\Rr_{\star}^d$ is  $\Rr^d$ where $d-1$ coordinate planes
$\{\bfx:x_j=0\},\,j\in [d-1],$ are removed.
Thus, $\xi_J(L,\bfx),\,J\subseteq [d-1],\,L\in\ZZZ_J,$  form 
a partition of unity on $\Rr_{\star}^d$. 
\\

\textit{4.2. Fourier integrals and dyadic decompositions.}
For a function $F(\bfx),\,\bfx\in\Rr^d,$ 
we have the decomposition
\begin{equation}
	F(\bfx)=\sum\nolimits_{J\in [d-1]} \sum\nolimits_{L\in\ZZZ_J} F_J (L,\bfx),
	\quad	
	\bfx\in\Rr_{\star}^d \, ,
	\label{eq6.0016**b}
\end{equation}
where $F_J(L,\bfx)=F(\bfx)\xi_J (L,\bfx)$. Note that for each 
$\bfx\in\Rr_{\star}^d$, at most $3^{d-1}$ terms in each of the sums  
\eqref{eq6.0016**a} and \eqref{eq6.0016**b} do not vanish.

For simplicity of notation, we do not specify $\phi\in\frak{S}(\Rr^d)$ in the 
following formulas, assuming it to be fixed. In particular, we will write
for the lattice sums 
$$S(\,\Lambda_{\theta},\phi,\bfu, T\,)=S(\,\Lambda_{\theta},\bfu, T\,)\,\, 
\text{and}\,\,
S(\,\Lambda_{\theta},\phi,T\,)=S(\,\Lambda_{\theta}, T\,).$$

For a subset 
$J\subseteq [d-1],\,\bfq\in\Rr_{>0}^d$ and 
$\phi\in\frak{S}(\Rr^d)$,  define
the Fourier integral 
\begin{equation}
	\FFF_J(\bfx,\bfq)=\int\nolimits_{\Rr^d}\,\, 
	\frac{e^{-2i\pi\,\bfx\centerdot\bfy}}{y_1\dots 
	y_d}\,\,\Xi_J(\bfy)\,\,
	\phi (\bfq\circ\bfy)\,\,d\bfy,
	\quad	
	\bfx\in\Rr^d \, ,
	\label{eq6.17}
\end{equation}
where 
\begin{equation}
	\Xi_J(\bfy)=\xi(Y)\,\xi(y_d)=
	\prod\nolimits_{j\in J}\,\xi_0 (\,y_j\,)\, 
	\prod\nolimits_{j\in[d]\setminus J}\xi_{\infty} (\,y_j\,)
	\label{eq6.17*}
\end{equation}
and  $\bfq\circ\bfy=(q_1x_1,\dots ,q_dx_d)$ for any two vectors 
$\bfq,\,\bfy\in\Rr^d$.

The cofactor $\Xi_J(\bfy)$ vanishes in a neighborhood of 
the singularities of integrand in \eqref{eq6.17}. Hence $\FFF_J(\bfx,\bfq)$ is a 
rapidly 
decreasing function of 
$\bfx\in\Rr^d$ as the Fourier transform of a function of class
$\frak{S}(\Rr^d)$.

For $\nu>0$,\,a lattice $\Gamma\in\LLL_d,\,\bfu\in\Rr^d$ 
and $\bfq\in\Rr^d_{>0}$ we define
\begin{equation}
\FFF_J(\,\nu,\Gamma,\bfu, \bfq\,)=	\sum\nolimits_{\bfx\in\Gamma } 
	\,\, 
	\FFF_J\big(\,\nu (\bfx-\bfu),\bfq\,\big)
	\label{eq6.18*}
\end{equation}
and
\begin{equation}
	\FFF_J(\,\nu,\Gamma,\bfq\,)=\sup\nolimits_{\bfu\in\Rr^d}\, 
	|\,\FFF_J(\,\nu, \Gamma,\bfu,\bfq\,)\,| \, .
	\label{eq6.018*}	
\end{equation} 


\begin{proposition}
 Let a set of numbers 
$\theta = (\theta_1,\dots,\theta_{d-1})$  be
$\kappa$-multiplicatively approximable. 
Then, we have the following.

\textit{(i)}
The lattice sum \eqref{eq6.06*} satisfies the \textbf{dyadic decomposition}
\begin{equation}
S(\,\Lambda_{\theta},\bfu, T\,)=\sum\nolimits_{J\in [d-1]} \,
\sum\nolimits_{L\in\ZZZ_J} 
S_J(\,L,\bfu, T\,)
	\label{eq6.18}
\end{equation}
where
\begin{equation}
	S_J(\,L,\bfu, T\,)=
\FFF_J\big(\,\nu(\bfm(L)),\,
D^{-1}_{\bfm(L)}\,\Lambda^{\perp}_{\theta},\,D^{-1}_{\bfm(L)}
\bfu,\,T^{-1}\bfm(L)\,\big),
	\label{eq6.18*q}
\end{equation}	

\textit{(ii)}
The supremum \eqref{eq6.07} satisfies the bound
\begin{equation}
	S(\,\Lambda_{\theta}, T\,)\le\sum\nolimits_{J\in [d-1]}\, 
	\sum\nolimits_{L\in\ZZZ_J} 
	S_J(\,L,\Lambda_{\theta}, T\,)\, ,
	\label{eq6.18**}	
\end{equation}
where
\begin{equation*}
	S_J(\,L,\Lambda_{\theta},T\,)=\FFF_J(\,\nu(\bfm(L)),\,
	D^{-1}_{\bfm(L)}\,\Lambda^{\perp}_{\theta},\,T^{-1}\bfm(L)\,)\, .
\end{equation*}
In these formulas, $\bfm(L)$ is the vector \eqref{eq6.0012}, $\nu(\bfm(L))$ and
$D_{\bfm(L)}$ are its hyperbolic norm and equalizer, respectively.
\end{proposition}	
\begin{proof}
	Note that $\Lambda_{\theta}^{\natural}\in\Rr_{\star}^d$, because
all numbers	$\theta_{1},\dots,\theta_{d}$ are irrational. 	
Applying decomposition \eqref{eq6.0016**b} to
$F(\bfx,\bfu,T)=\frac{e^{2i\pi\,\bfu\centerdot\bfx}}{x_1\dots 
x_d}\,\,\phi\,(T^{-1}\bfx)$,	
we obtain
\begin{equation}
 	S(\,\Lambda_{\theta}, \bfu, 
 	T\,)=\sum\nolimits_{\bfx\in\Lambda_{\theta}^{\natural}}\, 
 	\sum\nolimits_{J\in [d-1]}\, 
 	\sum\nolimits_{L\in\ZZZ_J}\,F_J(\,L,\bfx,\bfu,T\,)\, ,
	\label{eq6.018**}	
\end{equation}
where $F_J(\,L,\bfx,\bfu,T\,)=F(\bfx,\bfu, T)\xi_J (L,X)$.
With the hypothesis of Proposition~4.1, the series \eqref{eq6.018**} 
converges absolutely and uniformly.
Indeed, since $\xi_J (L,X)$ is nonnegative, we have
\begin{align*}
	&|\sum\nolimits_{\bfx\in\Lambda_{\theta}^{\natural}}\, 
	\sum\nolimits_{J\in [d-1]}\, 
	\sum\nolimits_{L\in\ZZZ_J}\,F_J(\,L,\bfx,\bfu,T\,)|
	\notag
	\\
	&\le\sum\nolimits_{\bfx\in\Lambda_{\theta}^{\natural}}\, 
	\sum\nolimits_{J\in [d-1]}\, 
	\sum\nolimits_{L\in\ZZZ_J}\,|F(\,\bfx,\bfu,T\,)|\,\xi_J(L,X)
	\notag
	\\
	&=\sum\nolimits_{\bfx\in\Lambda_{\theta}^{\natural}}\, 
	\,|F(\,\bfx,\bfu,T\,)|\quad\text{(\, by \eqref{eq6.0016**a} and 
	\eqref{eq6.0016**b}\,)}
	\notag
	\\
	&=\sum\nolimits_{\bfx\in\Lambda_{\theta}^{\natural}}\, 
	\,\frac{1}{|x_1\dots 
		x_d|}\,\,|\phi\,(T^{-1}\bfx)|
	\notag
	\\
	&= \,\sum_{m_d\ne0}\, 
	\sum_{m_1,\dots,m_{d-1}\in\Zz}
	\frac{1}{|m_d\,(\theta_{l}m_d -m_l),\dots,(\theta_{l}m_d -m_l)|}\,\, |\varphi 
	(T^{-1}\bfm)|
	\notag
	\\
	&=O_{\theta,\phi}(\,T^{d+\kappa}\,)\,,\quad\text{(\, just like in
	\eqref{eq6.00a}\,)}
\end{align*}
where $\varphi(\bfm)=\phi(\,\theta_{l}m_d -m_l,\dots,\theta_{l}m_d 
-m_l,\,m_d\,)\in\frak{S}(\Rr^d)$.

Therefore, the order of summation in
 \eqref{eq6.018**} can be exchanged to write
\begin{equation*}
	S(\,\Lambda_{\theta}, \bfu, T\,)= 
	\sum\nolimits_{J\in [d-1]}\, 
	\sum\nolimits_{L\in\ZZZ_J}\,S_J(\,L,\Lambda_{\theta},\bfu,T\,)\, ,
\end{equation*}
where
\begin{equation}
	S_J(\,L,\Lambda_{\theta},\bfu, 
	T\,)=\sum\nolimits_{\bfx\in\Lambda_{\theta}^{\natural}}\, 
	\,F_J(\,L,\bfx,\bfu, T\,)\, .
	\label{eq6.0018**a}	
\end{equation}

The next observation is significant for our construction.
By \eqref{eq6.0016**aaa}, $F_J(\,L,\bfx,\bfu,T\,)$ as a function of $\bfx$ 
is supported in $\CCC_J(L)$. Each point 
$\bfx\in\Lambda_{\theta}^{\natural}\cap\CCC_J(L)$ satisfies $|x_d|\ge\mu(L)$
by Lemma~3.1 and the definition of dyadic minima. 
Recall that $\xi_{\infty}(x)=1$ for $|x|\ge1$. Therefore,  
\eqref{eq6.0018**a} can be written as
\begin{equation}
	S_J(\,L,\Lambda_{\theta},\bfu, 
	T\,)=\sum\nolimits_{\bfx\in\Lambda_{\theta}^{\natural}}\, 
	\,Q_J(\,L,\bfx,\bfu,T\,)\, ,
	\label{eq6.0018**aa}	
\end{equation}
where
\begin{align*} 
Q_J(\,L,\bfx,\bfu,T\,)&=F_J(\,L,\bfx,\bfu,T\,)\,
\xi_{\infty} (\,\mu(L)^{-1}\,x_d\,)
\notag
\\
&=\frac{e^{2i\pi\,\bfu\centerdot\bfx}}{x_1\dots 
	x_d}\,\,\xi_J (L,X)\,\,\xi_{\infty} 
	(\,\mu(L)^{-1}\,x_d\,)\,\,\phi\,(T^{-1}\bfx)\,.
\end{align*}
Furthermore, summation in \eqref{eq6.0018**aa} can be extended to  the entire 
lattice $\Lambda_{\theta}$:
\begin{equation}
	S_J(\,L,\Lambda_{\theta},\bfu, T\,)=\sum\nolimits_{\bfx\in\Lambda_{\theta}}\, 
	\,Q_J(\,L,\bfx,\bfu,T\,)\, ,
	\label{eq6.0018**aaa}	
\end{equation}
since $\xi_{\infty} (\,\mu(L)^{-1}\,x_d\,)$ vanishes in a neighborhood of
the plane $\{\,\bfx: x_d=0\,\}$.
Therefore, the Poisson summation formula, see \cite[Chap. VII, Sec. 2]{27}, can be 
applied to \eqref{eq6.0018**aaa}. This gives
\begin{equation}
	S_J(\,L,\Lambda_{\theta},\bfu, T\,)=\sum\nolimits_{\bfx\in\Lambda_{\theta}}\, 
	\,Q_J(\,L,\bfx,\bfu,T\,)=\sum\nolimits_{\bfx\in\Lambda_{\theta}^{\perp}}\, 
	\,\widetilde Q_J(\,L,\bfu-\bfx,T\,)\, ,
	\label{eq6.0018***aaa}	
\end{equation}
where
\begin{align*} 
	\,&\widetilde Q_J(\,L,\bfu-\bfx,T\,)=\int\nolimits_{\Rr^d}\,
	e^{2i\pi\,(\bfu-\bfx)\centerdot\bfz}\,Q_J(\,L,\bfz,\bfu,T\,)\,d\bfz
	=\int\nolimits_{\Rr^d}
	\frac{e^{2i\pi\,(\bfu-\bfx)\centerdot\bfz}}{z_1\dots z_d}\,\,
		\notag
		\\
		&\times\prod\nolimits_{j\in J}\,\xi_0 (\,2^{l_j+1}z_j\,)\, 
		\prod\nolimits_{j\in J'}\xi_{\infty} (2z_j)
		\,\,\xi_{\infty} 
	(\,\mu(L)^{-1}\,z_d\,)\,\,\phi\,(T^{-1}\bfz)\,d\bfz\,.
\end{align*}
In this integral we put 
$$2^{l_j+1}z_j=y_j,\,\text{if}\,\,j\in 
J,\,\,2z_j=y_j,\,\text{if}\,\,j\in 
J',\,\,\mu(L)^{-1}z_d=y_d.$$ 
By \eqref{eq6.00012} and \eqref{eq6.000012}, 
$(\bfu-\bfx)\centerdot\bfz=
\nu(\bfm(L))\,(D_{\bfm(L)}^{-1}\,\bfu-D_{\bfm(L)}^{-1}\,\bfx)\centerdot\bfy$. 
Therefore,
\begin{align*} 
	\,&\widetilde Q_J(\,L,\bfu-\bfx, T\,)
	=\int\nolimits_{\Rr^d}
	\frac{e^{2i\pi\,\nu(\bfm(L))
	\big(D_{\bfm(L)}^{-1}\bfu-D_{\bfm(L)}^{-1}\bfx\big)\,\centerdot\,\bfy}}{y_1\dots
	 y_d}\,\,
	\notag
	\\
	&\times\prod\nolimits_{j\in J}\,\xi_0 (y_j)\, 
	\prod\nolimits_{j\in [d]\setminus J}\xi_{\infty} (y_j) 
	\,\,\phi\,(T^{-1}\bfm(L)\circ\bfy)\,d\bfy
	\notag
	\\
	&=\FFF_J\big(\,\nu(\bfm(L))\,
	(D^{-1}_{\bfm(L)}\bfx-D^{-1}_{\bfm(L)}\,\bfu),T^{-1}\bfm(L)\,\big)
\end{align*}
Substituting this formula to \eqref{eq6.0018***aaa},
we obtain \eqref{eq6.18} and \eqref{eq6.18*q}.

The inequality \eqref{eq6.18**} follows from \eqref{eq6.18} and \eqref{eq6.018*}.
\end{proof}

\section{Fourier integrals and lattice sums}
\label{sec5}
In this section we will estimate the supremum norms \eqref{eq6.018*} 
of the lattice sums \eqref{eq6.18*} with arbitrary unimodular lattices.
\begin{proposition}
Let $\nu>0,$ a lattice $\Gamma\subset\LLL_d,\,\bfq \in \Rr^d_{>0}$ and  
$J\subseteq [d-1]$ Then, we have
\begin{equation}
\FFF_J(\,\nu,\,\Gamma,\,\bfq)\le C_{a,\phi}\,\big(\,\nu^{-d} 
+\lambda_1\,(\Gamma^{\perp})^{-d}\,\big)\,\, \Upsilon_J (\bfq),	
\label{eq6.18bbb}
	\end{equation}
	where
\begin{equation}
	 \Upsilon_J (\bfq)=
\prod\nolimits_{j\in J}\frac{1}{(1+q_j)^a}\,\,
\prod\nolimits_{j\in [d]\setminus J}\frac{\log\, (\,2+q_j^{-1})\,}{(1+q_j)^a}	
	\label{eq6.18bbbb}
\end{equation}
with an arbitrary fixed $a>1$ and a constant $C_{a,\phi}$ depending only on $a$ and 
$\phi\in\frak{S}$.	
	\end{proposition}
	\begin{proof} Let us introduce the following Fourier integral
\begin{equation}
	\HHH_{J}(\bfx, \bfq)=\int\nolimits_{\Rr^d}\,\, 
	\frac{e^{-2i\pi\,\bfx\centerdot\bfy}}{y_1\dots 
		y_d}\,\,\Xi_J(\bfy)\,\,
	\psi\, (\bfq\circ\bfy)\,\,d\bfy,
	\quad	
	\bfx\in\Rr^d \, ,
	\label{eq6.18*bw}
\end{equation}
where $\Xi_J(\bfy)$	is defined in \eqref{eq6.17*} and 
\begin{equation*}
\psi\, (\bfq\circ\bfy)=\prod\nolimits_{j\in [d]} (1\,+\,q_j^2\,\, |y_j|^2)^{-a/2}	
	\end{equation*}
where $a>1$ is a fixed number. 
	
	$\HHH_{J}(\bfx,\bfq)$ as a function of $\bfx$ is the Fourier transform of a 
	$\CCC^{\infty}$	
	function, such that it and all its derivatives are integrable. Hence, 
	$\HHH_{J}(\bfx,\bfq)$ is a continuous function and 
	$\HHH_{J}(\bfx,\bfq)=O_a(|\bfx|_{\infty}^{-A})$ as $\bfx\to\infty,$ with an 
	arbitrary 	large $A>0$. Define 
\begin{equation*}
	\HHH_{J}(\,\nu,\Gamma, \bfu,\bfq\,)=\sum\nolimits_{\bfx\in\Gamma } 
	\,\, 
	\HHH_{J}\big(\,\nu (\bfx-\bfu), \bfq\,\big)
\end{equation*}
and
\begin{equation*}
	\HHH_{J}(\,\nu,\Gamma,\bfq\,)=\sup\nolimits_{\bfu\in\Rr^d}\, 
	|\,\HHH_{J}(\,\nu, \Gamma, \bfu,\bfq\,)\,| \, .
\end{equation*} 	
	
\begin{lemma} We have
\begin{equation*}
	\FFF_J(\,\nu,\,\Gamma,\bfq\,)\le 
	\GGG\,\,\HHH_{J}(\,\nu,\Gamma,\bfq\,)\,,
\end{equation*}
where
\begin{equation*}
\GGG=\int\nolimits_{\Rr^d}\,
\Big|\int\nolimits_{\Rr^d}\,e^{-2i\pi\,\bfx\centerdot\bfy}\,\,
\prod\nolimits_{j\in [d]} (1\,+\, |y_j|^2)^{a/2}
\,\,\phi(\bfy)\,\,d\bfy\, \Big|\,d\bfx <\infty\,.
\end{equation*}	
	\end{lemma}	
\begin{proof} Let us put
\begin{equation*}
	\FFF_J(\bfx,\bfq)=\int\nolimits_{\Rr^d}\,\, 
	\frac{e^{-2i\pi\,\bfx\centerdot\bfy}}{y_1\dots 
		y_d}\,\,\Xi_J(\bfy)\,\,\psi\,(\bfq\circ\bfy)\,\,\psi\,(\bfq\circ\bfy)^{-1}
	\phi(\bfq\circ\bfy)\,\,d\bfy\,.
\end{equation*}
Write this Fourier integral as a convolution:	
\begin{equation}
	\FFF_J(\bfx,\bfq)=\int\nolimits_{\Rr^d}\,\, \HHH_{J}(\,\bfx-\bfy,\bfq\,)\,\,
	\GGG(\,\bfy,\bfq\,)\,\,d\bfy
	\label{eq6.18ww}
\end{equation}
where
\begin{equation*}
	\GGG(\,\bfy,\bfq\,)=\int\nolimits_{\Rr^d}\,\, e^{-2i\pi\,\bfy\centerdot\bfz}
		\,\,\psi\,(\bfq\circ\bfz)^{-1}\,\,
	\phi(\bfq\circ\bfz)\,\,d\bfz,
\end{equation*}
Replacing $\bfx$ with $\nu\bfx-\nu\bfu$ in \eqref{eq6.18ww} and summing over 
$\bfx\in\Gamma$, we obtain
\begin{equation*}
	\FFF_J(\,\nu,\Gamma,\bfu,\bfq\,)=\int\nolimits_{\Rr^d}\,\, 
	\HHH_{J}(\,\nu,\Gamma,-\nu\bfu-\bfy,\,\bfq\,)\,\,
	\GGG_{J}(\bfy,\bfq)\,\,d\bfy\, .
\end{equation*}
Now, we have
\begin{equation*}
	\FFF_J(\,\nu,\,\Gamma,\bfq\,)\le 
	\GGG(\bfq)\,\,\HHH_{J}(\,\nu,\Gamma,\bfq\,)
\end{equation*}
where
\begin{equation*}
	\GGG(\bfq)=\int\nolimits_{\Rr^d}\,
	\Big|\int\nolimits_{\Rr^d}\,\, e^{-2i\pi\,\bfy\centerdot\bfz}
	\,\,\psi\,(\bfq\circ\bfz)^{-1}\,
	\phi(\bfq\circ\bfz)\,\,d\bfz\,\Big|\, d\bfy = \GGG\, ,
\end{equation*}
because this integral does not depend on $\bfq\in\Rr_{>0}^d$.
\end{proof}	

The integral \eqref{eq6.18*bw} is equal to the product of one-dimensional 
integrals:	
\begin{equation}		
\HHH_{J}(\bfx,\bfq)=\prod\nolimits_{j\in J}\,h_0(x_j,q_j)	\,
\prod\nolimits_{j\in [d]\setminus J}\,h_{\infty}(x_j,q_j)\, ,	
\label{eq6.18*b}
\end{equation}		
where		
\begin{align*}
	h_0(x,q)&=\int\nolimits_{\Rr}\,\, 
	\frac{e^{-2i\pi\,xy}}{y}\,\,\xi_0(y)\,\,
	(1\,+\,q^2\,\, |y|^2)^{-a/2}\,\,d y\, ,
	\notag
	\\
	h_{\infty}(x,q)&=\int\nolimits_{\Rr}\,\, 
	\frac{e^{-2i\pi\,xy}}{y}\,\,\xi_{\infty}(y)\,\,
	(1\,+\,q^2\,\, |y|^2)^{-a/2}\,\,d y\,.
\end{align*}

\begin{lemma} For $x\in\Rr$ and $q>0$, we have
\begin{align*}
	|\,h_0(x,q)\,|&\le C_{a,A}\,\,\frac{1}{(1\,+\,q\,)^{a}}\,\,
	(1\,+\,|x|\,)^{-A},
	\\
	|\,h_{\infty}(x,q)\,|&\le C_{a,A}\,\,\frac{\log\,(2+q^{-1})}
	{(1\,+\,q\,)^{a}}\,\,
	(1\,+\,|x|\,)^{-A}
\end{align*}
with an arbitrary large $A>1$.
\end{lemma}
\begin{proof}
Recall that $\xi_0(x)$ and $\xi_{\infty}(x)$ are  
$\CCC^{\infty}$ functions supported on
$$[-1,-2^{-3}]\cup[2^{-3},1]\quad \text{and}\quad 
(-\infty, -2^{-2}]\cup [2^{-2},\infty),$$ respectively, and  $\xi_{\infty}(x)=1$ 
for $|x|\ge1$. 

Let us prove the first bound. 
Integrating 
$K\ge 0$ times by parts, we obtain
\begin{align*} 
	 (2i\pi\,x)^K\,h_0(x,q)&=\int\nolimits_{\Rr}\,\, 
	e^{-2i\pi\,xy}\,\,\big(\frac{d}{dy}\big)^K\,\Big(
	\frac{\xi_0(y)}{y}\,
	(1\,+\,q^2\,\, |y|^2)^{-a/2}\Big)\,\,d y
	\notag
	\\
	&=\sum\nolimits_{k=0}^A \,\frac{K!}{k!(K-k)!}\,\,E_k(x,q)\, ,
\end{align*}
where
\begin{align*}	
E_k(x,q)	=
	\int\nolimits_{\Rr}\,\, 
	e^{-2i\pi\,xy}\,\,\Big(
	\frac{\xi_0(y)}{y}\,\Big)^{(K-k)}\,
	\Big((1\,+\,q^2\,\, |y|^2)^{-a/2}\Big)^{(k)}\,\,d y.
\end{align*}
As usually, $f^{(k)}$ denotes the $k$-th derivative of $f$.
For these integrals, we have 
\begin{equation*}
	|\,E_k(x,q)\,|\ll
	\int\nolimits_{2^{-3}}^1
	\,\frac{\max\,\{q^k,1\}}{(1\,+\,q\,y\,)^{a+k}}\,dy
	\ll\frac{1}{(1\,+\,q\,)^{a}}\,.
\end{equation*}
Therefore,
\begin{equation*}
(1+	|x|^K)\,|\,h_0(x,q)\,|\ll\frac{1}{(1\,+\,q\,)^{a}}\,.
\end{equation*}
and the required bound
follows.

Let us prove the second bound.
Integrating 
$K\ge 0$ times by parts, we obtain 
\begin{align*} 
	(2i\pi\,x)^K\,&h_{\infty}(x,q)=\int\nolimits_{\Rr}\,\, 
	e^{-2i\pi\,xy}\,\,\big(\frac{d}{dy}\big)^K\,\Big(
	y^{-1}\,\xi_{\infty}(y)\,
	(1\,+\,q^2\,\, |y|^2)^{-a/2}\Big)\,\,d y
	\notag
	\\
	&=\sum_{k_1+k_2+k_3=K} \frac{K!}{k_1!\,k_2!\,k_3!} 
	\,\,E_{k_1,k_2,k_3}(x,q)\, ,
\end{align*}
where
\begin{align*}
E_{k_1,k_2,k_3}&(x,q)
\notag
\\
&=\int\nolimits_{\Rr}\,\, 
e^{-2i\pi\,xy}\,\,
(y^{-1})^{(k_1)}\,\,\xi_{\infty}(y)^{(k_2)}\,\,
\big(\,(1\,+\,q^2\,\, |y|^2)^{-a/2}\,\big)^{(k_3)}dy\,.
\end{align*}
For these integrals we distinguish  the following cases. 
\\
\textit{(i)} For $k_1=k_2=0,\, k_3\ge0$,
\begin{align*}
	|\,E_{0,0,k_3}(x,q)\,|&\ll \int\nolimits_{2^{-2}}^{\infty}
	\,\,\frac{\max\,\{q^{k_3},1\}}{y\,(1\,+\,q\,y\,)^{a+k_3}}\,dy
	\notag
	\\
	&= \int\nolimits_{2^{-2}q}^{\infty}
	\,\,\frac{\max\,\{q^{k_3},1\}}{y\,(1\,+\,y\,)^{a+k_3}}\,dy
	\ll\frac{\log\,(2+q^{-1})}{(1\,+\,q\,)^{a}}\,.
\end{align*}
\textit{(ii)} For $k_1=0,\, k_2>0,\,k_3\ge0 $,
\begin{equation*}
	|\,E_{0,k_2,k_3}(x,q)\,|\ll
	\int\nolimits_{2^{-2}}^1
	\,\frac{\max\,\{q^k,1\}}{(1\,+\,q\,y\,)^{a+k}}\,dy
	\ll\frac{1}{(1\,+\,q\,)^{a}}\,.
\end{equation*}
\textit{(iii)} For $k_1>0,\, k_2=0,\,k_3\ge0 $,
\begin{align*}
	|\,E_{k_1,0,k_3}(x,q)\,|&\ll \int\nolimits_{2^{-2}}^{\infty}
	\,\,\frac{\max\,\{q^{k_3},1\}}{y^{k_1+1}\,(1\,+\,q\,y\,)^{a+k_3}}\,dy
	\ll\frac{1}{(1\,+\,q\,)^{a}}\,.
\end{align*}
\textit{(iv)} For $k_1>0,\, k_2>0,\,k_3\ge0 $,
\begin{align*}
	|\,E_{k_1,0,k_3}(x,q)\,|&\ll \int\nolimits_{2^{-2}}^{1}
	\,\,\frac{\max\,\{q^{k_3},1\}}{y^{k_1+1}\,(1\,+\,q\,y\,)^{a+k_3}}\,dy
	\ll\frac{1}{(1\,+\,q\,)^{a}}\,.
\end{align*}
All possible cases are considered. As a result,  we have
\begin{equation*}
	(1+	|x|^K)\,|\,h_{\infty}(x,q)\,|\ll
	\frac{\log\,(2+q^{-1})}{(1\,+\,q\,)^{a}}\,.
\end{equation*}
This proves the second bound.
\end{proof}

By lemma~5.2 we obtain the following bound for the product \eqref{eq6.18*b}		
\begin{align*}		
	\HHH&_{J}(\bfx,\bfq)
	\notag
	\\
	&\ll\prod\nolimits_{j\in[d]}(1\,+\,|x_j|\,)^{-A}\prod\nolimits_{j\in 
	J}\,\frac{\log\,(2+q_j^{-1})}{(1\,+\,q_j\,)^{a}}	\,
	\prod\nolimits_{j\in [d]\setminus 
	J}\,\frac{\log\,(2+q_j^{-1})}{(1\,+\,q_j\,)^{a}}
	\notag
	\\
	&\ll(1\,+\,|\bfx|_{\infty}\,)^{-A}\,\,\Upsilon_J (\bfq)\, .	
\end{align*}	
Replacing here $\bfx$ with $\nu(\bfx-\bfu)$, summing over $\bfx\in\Gamma$ 
and using Proposition~2.1,
we obtain the following bound
\begin{equation*}
\HHH_{J}(\,\nu,\Gamma, \bfq\,)\ll \big(\,\nu^{-d} 
+\lambda_1\,(\Gamma^{\perp})^{-d}\,\big)\,\,\Upsilon_J (\bfq)\, .	
\end{equation*}
       This inequality together with Lemma~5.1 implies \eqref{eq6.18bbb}.
	\end{proof}
		
\section{Proof of main bounds for lattice sums}
\label{sec6}		

We first specialize Proposition~5.1 for 
$\bfq=T^{-1}\,\bfm(L),\,\nu=\nu(\bfm(L)),\, 
\Gamma=D^{-1}_{\bfm(L)}\,\Lambda_{\theta}^{\perp}$, see \eqref{eq6.0012} and 
\eqref{eq6.000012}.
Notice also the duality relation 
$D_{\bfm(L)}^{-1}\,\Lambda_{\theta}^{\perp}=
(D_{\bfm(L)}\,\Lambda_{\theta})^{\perp}\,.$ 
\begin{proposition}
	Let a set of numbers 
	$\theta = (\theta_1,\dots,\theta_{d-1})$  be
	$\kappa$-multiplicatively approximable.  Then	
	\begin{align}
		\FFF_J(\,\nu(\bfm(L)),&\,D^{-1}_{\bfm(L)}\,
		\Lambda^{\perp}_{\theta},\,T^{-1}\bfm(L)\,)
		\notag
		\\
		&\le\,C(\kappa)\,\, 
		2^{\frac{\kappa}{\kappa +1}\, |L|_1}\,\,
		\frac{\,\big(\log T+|L|_1)^{|J'|}\,\log T}{(\,1+T^{-1}\,
			2^{\frac{1}{\kappa +1}\, |L|_1}\,\big)^a}
		\label{eq6.a012**}
	\end{align}
with an arbitrary fixed $a>1$ and a constant $C(\kappa)$ depending only on the 
constant $c(\kappa)$ in \eqref{eq6.02*}. 	
	\end{proposition}
\begin{proof}
For the indicated parameters the inequality in Proposition~5.1 takes the form
	\begin{align}
		\FFF_J(\,\nu(\bfm(L)),&\,D^{-1}_{\bfm(L)}\,
		\Lambda^{\perp}_{\theta},\,T^{-1}\bfm(L),\,)
		\notag
		\\
		&\le C_{a,\phi}\,\big(\,\nu(\bfm(L))^{-d} 
		+\lambda_1\,(D_{\bfm(L)}\,
		\Lambda_{\theta})^{-d}\,\big)\,\, \Upsilon (\,T^{-1}\bfm(L)\,),	
		\label{eq6.180bbb}
	\end{align}
	where
	\begin{align}
		\Upsilon (&T^{-1}\bfm(L))
		\notag
		\\
		&=\prod\nolimits_{j\in J}\frac{1}{(1+T^{-1}m_j)^a}\,\,
		\cdot\prod\nolimits_{j\in J'}\frac{\log\, 
		(\,2+T m_j^{-1})\,}{(1+T^{-1}m_j)^a}\,\cdot
		\frac{\log\, 
			(\,2+Tm_d^{-1})\,}{(1+T^{-1}m_d)^a}		
		\label{eq6.018bbbb}
	\end{align}
	with  $a>1$. Here $m_j=2^{-l_j-1},\,j\in [d-1],\,\,m_d=\mu(L),\, 
	L=(l_1,\dots,l_d)\in\ZZZ_J.$ 
	Notice that $m_j,\,j\in [d-1],$ are ``small'', and
	$m_d$ is ``large''.
	Using Proposition~3.1, we estimate the product \eqref{eq6.018bbbb} as follows
	\begin{align*}
		\Upsilon (T^{-1}\,\bfm(L))
		&<\prod\nolimits_{j\in J'}\,\log\, 
			(\,2+T\, m_j^{-1}\,)\cdot
		\frac{\log 
			(\,2+T\,m_d^{-1}\,)}{(\,1+T^{-1}\,m_d \,)^a}	
			\notag
			\\
	&=\prod\nolimits_{j\in J'}\,\log \, 
	(\,2+T \, 2^{l_j+1}\,)\cdot
	\frac{\log (\, 2+T\,\mu(L)^{-1}\,)}{(\,1+T^{-1}\,\mu(L)\,)^a}
	   \notag
	   \\
	&\ll\prod\nolimits_{j\in J'}\,\log \, 
	(\,2+T \, 2^{|L|_1}\,)\cdot
	\frac{\log (\, 2+T\,2^{\frac{-1}{\kappa +1}\, 
	|L|_1}\,)}{(\,1+T^{-1}\,2^{\frac{1}{\kappa +1}\, |L|_1}\,)^a} 
        \notag
        \\
        &\ll\frac{\,\big(\log T+|L|_1\,\big)^{|J'|}\,\log T}
        {\big(\,1+T^{-1}\,2^{\frac{1}{\kappa +1}\, 
        		|L|_1}\,\big)^a}\,.	
	\end{align*}
	
On the other hand, by Proposition~3.1, we also have
\begin{equation*}
\nu(\bfm(L))^{-d} 
+\lambda_1\,(D_{\bfm(L)}\,
\Lambda_{\theta})^{-d}\, \ll\, 2^{\frac{\kappa}{\kappa +1}\, |L|_1}
	\end{equation*}
	
	Substituting these bounds into 
	\eqref{eq6.180bbb}, we obtain \eqref{eq6.a012**}.	
	\end{proof}
	
	To complete the proof of Theorem~1.2 we will apply Proposition~6.1 to 
	Proposition~4.1. 
	Let $T>4$ and $s=\lfloor (\kappa+1)\log T\rfloor >0$, so that
	$$2^{\frac{s}{\kappa +1}}\le T< 2^{\frac{s+1}{\kappa +1}}\,.$$
	We also have
	\begin{equation*}
		\#\,\{\,L\in\ZZZ_J: |L|_1=\sum\nolimits_{j\in J}\, l_j=n\,\}
		\le c_J (\,n+1\,)^{|J|-1}\, .	
	\end{equation*}	
	Substituting \eqref{eq6.a012**} with $a>\kappa$ into \eqref{eq6.18**}, we obtain
\begin{align*}
S(\,\Lambda_{\theta}, T\,)&\le\sum\nolimits_{J\in [d-1]}\, 
\sum\nolimits_{L\in\ZZZ_J} 
\FFF_J(\,\nu(\bfm(L)),\,
D^{-1}_{\bfm(L)}\,\Lambda^{\perp}_{\theta}, \,T^{-1}\bfm(L)\,)
\notag
\\
&\ll\sum\nolimits_{J\in [d-1]}\,\sum\nolimits_{n\ge0}\,\, 2^{\frac{\kappa}{\kappa 
+1}\, n}\,\,
\frac{\,(n+1)^{|J|-1}\,(s+n)^{|J'|}\,s}{\big(\,1+\,
	2^{\frac{1}{\kappa +1}\, (n-s)}\,\big)^a}
	\notag
	\\
&=\sum\nolimits_{J\in [d-1]}\,\Big(\,\sum\nolimits_{0\le n\le s} 
+ \sum\nolimits_{n>s}\,\Big)\, .
\end{align*}
Taking into account that $|J|+|J'|=d-1$, we obtain for the first sum
\begin{align*}
\sum\nolimits_{0\le n\le s}\,	\ll\, s^{d-1}\,
\sum\nolimits_{0\le n\le s}\, 2^{\frac{\kappa}{\kappa +1}\, 
n}\,\ll \,2^{\frac{\kappa}{\kappa +1}\, s}\,\,s^{d-1}\,.
\end{align*}
For the second, we find
\begin{align*}
	\sum\nolimits_{n>s}
	&\ll\sum\nolimits_{n\ge s}\,\, 2^{\frac{\kappa}{\kappa +1}\, n}\,\,
	\frac{\,(n+1)^{|J|-1}\big(s+n)^{|J'|}\,s}{(\,1+\,
		2^{\frac{1}{\kappa +1}\, (n-s)}\,\big)^a}
	\notag
	\\
	&\ll\sum\nolimits_{n\ge s}\,\, 2^{\frac{\kappa}{\kappa +1}\, n-\frac{a}{\kappa 
	+1}\, (n-s)}\,\,
	\,(n+1)^{|J|-1}\big(s+n)^{|J'|}\,\,s
	\notag
	\\
	&\ll 2^{\frac{\kappa}{\kappa +1}\,s}\,\sum\nolimits_{m\ge 0}\,\,
	2^{\frac{\kappa -a}{\kappa +1}\, m}\,\, (m+s)^{d-2}\,\,s
	\notag
	\\
	&\ll 2^{\frac{\kappa}{\kappa +1}\,s}\,\, s^{d-1}\,.
\end{align*}
As a result, we have
\begin{equation*}
S(\,\Lambda_{\theta}, T\,)\ll 2^{\frac{\kappa}{\kappa +1}\,s}\,\, s^{d-1}
\ll T^{\kappa}\,\, (\log T)^{d-1} \,.
\end{equation*}

The proof of Theorem~1.2 is completed.

\enlargethispage{4\baselineskip}


\begin{thebibliography}{99}
	


	
	
	
		\bibitem{4***} 
	J.W.S.Cassels,
	\emph{An introduction to the geometry of numbers}, 
	Springer Verlag, Berlin, 1959.
	
	
	\bibitem{12}
	P.M.Gruber, C.G.Lekkerkerker,
	\emph{Geometry of numbers}, North--Holland Math. Libraey, vol. 37, Elsevier 
	Sci. Pub., Amsterdam, 1987.
	
	
	
	\bibitem{16}
	D.Y.Kleinbock,
	\emph{Ergodic Theory on Homogeneous Spaces and Metric Number Theory}, in
	Mathematics of Complexity and Dynamical Systems (Ed. R.A.Meyers), 
	Springer, New York, 2011.
	
	
	
	
	\bibitem{4****} 
	L.Kuipers, H.Niederreiter,
	\emph{Uniform distribution of sequences}, 
	John Wiley \& Sons, N.Y., 1974.
	
	
	\bibitem{16a}
	G.A.Margulis,  
	\emph{Diophantine approximation, lattices and ﬂows on homogeneous spaces}, in 
	A panorama of number theory or the view from Baker’s garden (Ed. G.Wüstholz), 
	pp. 280--310, Cambridge Univ. Press, 2002.
	
	
	
	
    
    \bibitem{19} 
	W.M.Schmidt,
	\emph{Diophantine Approximation}, Lect. Not. in Math., vol. 785,
	Springer--Verlag, Berlin, 1980.
	

	\bibitem{21} 
	M.M.~Skriganov,
	\emph{Constructions of uniform distributions in terms of geometry numbers}, 
	Algebra i Analiz, {\bf 6}(3), (1994), 200--230; reprinted in St.Petersburg 
	Math. J., {\bf 6}(3), (1995), 635--664.
	
    
    \bibitem{22} 
    M.M.~Skriganov,
    \emph{Ergodic theory on $SL(n)$, Diophantine approximations and anomalies 
    in the lattice point problem}, 
    Inventiones Math., {\bf 132}(1), (1998), 1--72.
      
   
   \bibitem{22*} 
   M.M.~Skriganov,
   \emph{Integer points in a simplex 
   	and related Diophantine problems: 
   Hardy--Littlewood asymptotics in higher dimensions}, (to appear).
   
    
    \bibitem{24} 
    D.C.Spencer,
    \emph{The lattice points of tetrahedra}, 
    J. Math. Phys. Mass. Inst. Techn., {\bf 21}, (1942), 189--197;
    reprinted in D.C.Spencer, SELECTA, vol.1, pp. 24--32, World Sci. Pub., 1985.
    
    
    \bibitem{25} 
     V.G.Sprind\u{z}uk,
     \emph{Metric theory of Diophantine approximations}, John Wiley and Sons, 
     NY, 1979.


     \bibitem{26} 
     A.N.Starkov,
     \emph{Dynamical systems on homogeneous spaces},  Transl. Math. Monogr., 
     vol. 190, AMS,  Providence, 2000.


     \bibitem{27} 
     E.M.Stein, G.Weiss,
     \emph{Introduction to Fourier analysis on Euclidean spaces}, Princeton 
     Univ. Press, Princeton, 1971.

\end{thebibliography}
\end{document}